\documentclass[12 pt]{article}
\usepackage[pdftex]{graphicx}
\usepackage{amsmath}
\usepackage{mathtools}
\usepackage{amsthm}
\usepackage{amsfonts}
\usepackage{amssymb}
\usepackage[margin=1.0in]{geometry}
\usepackage{tikz}
\allowdisplaybreaks

\newtheorem{defin}{Definition}[section]
\newtheorem{theorem}[defin]{Theorem}
\newtheorem{lemma}[defin]{Lemma}
\newtheorem{prop}[defin]{Proposition}
\newtheorem{cor}[defin]{Corollary}

\newtheorem{que}[defin]{Question}

\newcommand{\fr}{Fra\"iss\'e }
\renewcommand{\phi}{\varphi}

\newcommand{\aut}[1]{\mathrm{Aut}(#1)}

\begin{document}
\title{Amenability and Unique Ergodicity of Automorphism Groups of Fra\"iss\'e Structures}
\author{Andy Zucker}
\date{December 2013}
\maketitle

\begin{abstract}
In this paper we consider those \fr classes which admit companion classes in the sense of [KPT]. We find a necessary and sufficient condition for the automorphism group of the \fr limit to be amenable and apply it to prove the non-amenability of the automorphism groups of the directed graph $\mathbf{S}(3)$ and the Boron tree structure $\mathbf{T}$. Also, we provide a negative answer to the Unique Ergodicity-Generic Point problem of Angel-Kechris-Lyons [AKL]. By considering $\mathrm{GL}(\mathbf{V}_\infty)$, where $\mathbf{V}_\infty$ is the countably infinite dimensional vector space over a finite field $F_q$, we show that the unique invariant measure on the universal minimal flow of $\mathrm{GL}(\mathbf{V}_\infty)$ is not supported on the generic orbit. \let\thefootnote\relax
\let\thefootnote\relax
\end{abstract}

\section{Introduction}

Let $G$ be a Hausdorff topological group. $G$ is said to be \textbf{amenable} if every jointly continuous action on a compact Hausdorff space $X$ (often called a $G$-flow) supports an invariant Borel probability measure. It is normally quite difficult to determine wheter a given topological group $G$ is amenable or not. However, there are some observations we can make to simplify our discussion. If $X$ is a $G$-flow, a \emph{subflow} $Y\subseteq X$ is a compact subspace which is invariant under $G$-action. Using Zorn's lemma, we see that each $G$-flow contains a minimal subflow. Therefore to show that $G$ is amenable, it is enough to show that every minimal $G$-flow supports an invariant Borel probability measure. We can simplify things even further; if $X$ and $Y$ are $G$-flows, we say that $\pi: X\rightarrow Y$ is a $G$-map if $\pi$ is continuous and $\pi(g\cdot x) = g\cdot \pi(x)$ for each $x\in X$. It is a fact that each topological group $G$ admits, up to isomorphism, a \textbf{universal minimal flow} $M(G)$, a flow which is minimal and such that for any other minimal $G$-flow $X$, there is a $G$-map $\pi: M(G)\rightarrow X$ (see [A] for further reading). Hence by push forward, $G$ is amenable if and only if the flow $M(G)$ supports an invariant Borel probability measure.

Kechris, Pestov, and Todorcevic in [KPT] provided a way of explicitly describing $M(G)$ in many instances where $G$ is the automorphism group of a \fr structure. We will discuss this in much more detail in section 2. In particular, having a concrete representation of $M(G)$ allows one to address questions of amenability. Kechris and Soki\'c in [KS] take advantage of this to show that the automorphism groups of the random poset and the random distributive lattice are not amenable.

Indeed, amenability is not the only property of topological groups $G$ for which is it sufficient to verify a property only for the flow $M(G)$. $G$ is said to be \textbf{uniquely ergodic} if each minimal $G$-flow supports a unique invariant Borel probability measure, and $G$ is said to have the \textbf{generic point property} if for each minimal flow $X$, there is $x\in X$ with $G\cdot x$ comeager in $X$. It is worth noting that any comeager orbit must be unique since the intersection of two comeager sets is comeager, hence nonempty. As promised, $G$ is uniquely ergodic iff $M(G)$ supports a unique invariant Borel probability measure, and $G$ has the generic point property iff there is $x\in M(G)$ with $G\cdot x$ comeager. These two definitions provide us with two notions of a canonical choice of ``large'' subset of $M(G)$, namely the support of the unique measure and the comeager orbit. In all examples of groups previously shown to be uniquely ergodic and to have the generic point property, the measure was in fact supported on the generic orbit. Angel, Kechris, and Lyons in [AKL] asked whether this phenomenon holds generally.

This paper is divided into two main parts. The first half generalizes the methods in [KS] and finds a necessary and sufficient condition for those groups whose universal minimal flows can be described using the methods of [KPT] to be amenable. Armed with this condition, we consider the ultrahomogeneous directed graph $\mathbf{S}(3)$ and the Boron tree structure $\mathbf{T}$, which are defined in sections 4 and 5, respectively. We show that:

\begin{theorem}
The groups $\aut{\mathbf{S}(3)}$ and $\aut{\mathbf{T}}$ are not amenable. 
\end{theorem}

The second half answers in the negative the question of Angel, Kechris, and Lyons by exhibiting the following counterexample; fix a finite field $F_q$, and let $\mathbf{V}_\infty$ be the countably infinite vector space over $F_q$. In section 6, we prove:
\begin{theorem}
The unique invariant measure on $M(\aut{\mathbf{V}_\infty})$ is not supported on the generic orbit. 
\end{theorem} 
Of course, after asserting where the unique measure in this example is \emph{not} supported, it is natural to ask where it is supported. This is the focus of sections 7 and 8. 

\subsection*{Acknowledgements}
I would like to thank Alekos Kechris, Miodrag Soki\'c, and Anush Tserunyan for many helpful discussions and comments, as well as Ola Kwiatkowska and Robin Tucker-Drob for patiently lending an ear. I would also like to thank an anonymous referee for many helpful comments and suggestions. This research was supported by the Caltech SURF Office and the Logic in Southern California NSF grant EMS W21-RTG.

\section{Background}
This section describes the necessary model theoretic background as well as a brief summary of the material from [KPT].

A \emph{language} $L = \{R_i\}_{i\in I}\cup \{f_j\}_{j\in J}\cup \{c_k\}_{k\in K}$ is a set of relation, function, and constant symbols. Each relation and function symbol has an arity $n(i), m(j)\in \mathbb{N}^+$ for $i\in I, j\in J$. In this paper, all languages are assumed to be countable. An \emph{$L$-structure} $\mathbf{A} = \langle A, \{R^{\mathbf{A}}_i\}_{i\in I}, \{f^{\mathbf{A}}_j\}_{j\in J}, \{c^{\mathbf{A}}\}_{k\in K}\rangle$ consists of a set $A\neq\emptyset$ called the \emph{universe} of $\mathbf{A}$, $R^{\mathbf{A}}_i\subseteq A^{n(i)}$, $f^{\mathbf{A}}_j: A^{m(j)}\rightarrow A$, and $c^{\mathbf{A}}_k\in A$. An \emph{embedding/isomorphism} $\pi: \mathbf{A}\rightarrow \mathbf{B}$ of $L$-structures is an injective/bijective map which preserves structure: $R^{\mathbf{A}}_i(a_1,...,a_{n(i)})\Leftrightarrow R^{\mathbf{B}}_i(\pi(a_1),...,\pi(a_{n(i)}))$, $\pi(f^{\mathbf{A}}_j(a_1,...,a_{m(j)})) = f^{\mathbf{B}}_j(\pi(a_1),...,\pi(a_{m(j)}))$, and $\pi(c^{\mathbf{A}}_k) = c^{\mathbf{B}}_k$. We write $\mathbf{A}\leq \mathbf{B}$ if there is an embedding from $\mathbf{A}$ into $\mathbf{B}$, and we write $\mathbf{A}\cong \mathbf{B}$ if there is an isomorphism between the two. We say $\mathbf{A}$ is a \emph{substructure} of $\mathbf{B}$, written $\mathbf{A}\subseteq \mathbf{B}$, if $A\subseteq B$ and $A$ is closed under the functions $f_j^{\mathbf{B}}$. If $L_0\subseteq L$ and $\mathbf{A}$ is an $L$-structure, we write $\mathbf{A}|_{L_0}$ for the structure obtained by taking $\mathbf{A}$ and ignoring the interpretations of relation, function, and constant symbols in $L\backslash L_0$. Similarly, if $\mathcal{K}$ is a class of $L$-structures, we write $\mathcal{K}|_{L_0}$ for the class $\{\mathbf{A}|_{L_0}: \mathbf{A}\in \mathcal{K}\}$.

Let $\mathcal{K}$ be a class of finite $L$-structures closed under isomorphism with countably many isomorphism types and such that there are structures in $\mathcal{K}$ of arbitrarily large finite cardinality. We call $\mathcal{K}$ a \textbf{Fra\"iss\'e class} if the following three items hold:
\begin{itemize}
\item
Hereditary Property (HP): If $\mathbf{B}\in \mathcal{K}$ and $\mathbf{A}\leq \mathbf{B}$, then $\mathbf{A}\in \mathcal{K}$.
\item
Joint Embedding Property (JEP): If $\mathbf{A}, \mathbf{B}\in \mathcal{K}$, then there is $\mathbf{C}\in \mathcal{K}$ such that $\mathbf{A}, \mathbf{B}\leq \mathbf{C}$.
\item
Amalgamation Property (AP): If $\mathbf{A}, \mathbf{B}, \mathbf{C}\in \mathcal{K}$ and $f: \mathbf{A}\rightarrow \mathbf{B}$, $g: \mathbf{A}\rightarrow \mathbf{C}$ are embeddings, then there is $\mathbf{D}\in \mathcal{K}$ and embeddings $r: \mathbf{B}\rightarrow \mathbf{D}$, $s: \mathbf{C}\rightarrow \mathbf{D}$ such that $r\circ f = s\circ g$.
\end{itemize}
A \textbf{Fra\"iss\'e structure} is a countably infinite, locally finite (finitely generated substructures are finite) $L$-structure $\mathbf{K}$ which is \emph{ultrahomogeneous}, i.e.\ any isomorphism between finite substructures extends to an automorphism of $\mathbf{K}$. An equivalent definition is that a Fra\"iss\'e structure satisfies the \emph{Extension Property}: for any $\mathbf{A}\subseteq \mathbf{B}$, $\mathbf{A}, \mathbf{B}\in\mathcal{K}$ and any embedding $f: \mathbf{A}\rightarrow \mathbf{K}$, there is an embedding $g: \mathbf{B}\rightarrow \mathbf{K}$ extending $f$. For any infinite structure $\mathbf{X}$, let $\mathrm{Age}(\mathbf{X})$ denote the class of finite substructures which embed into $\mathbf{X}$. Fra\"iss\'e's theorem states that there is a one-to-one correspondence between Fra\"iss\'e classes and Fra\"iss\'e structures: the age of each Fra\"iss\'e structure is a Fra\"iss\'e class, and each Fra\"iss\'e class is the age of a Fra\"iss\'e structure unique up to isomorphism. For $\mathcal{K}$ such a class, we write $\mathrm{Flim}(\mathcal{K})$ for the associated structure, called the \textbf{Fra\"iss\'e limit} of $\mathcal{K}$. See Hodges [Ho] for a more detailed exposition.

For finite structures $\mathbf{A}\leq \mathbf{B}$, let $\binom{\mathbf{B}}{\mathbf{A}}$ denote those substructures of $\mathbf{B}$ which are isomorphic to $\mathbf{A}$. A class $\mathcal{K}$ of finite structures satisfies the \emph{Ramsey Property} (RP) if for any $\mathbf{A}\leq \mathbf{B}\in \mathcal{K}$ and any $k\geq 2$, there is $\mathbf{C}\in \mathcal{K}$, $\mathbf{B}\leq \mathbf{C}$ such that for any coloring $c: \binom{\mathbf{C}}{\mathbf{A}}\rightarrow k$, there is $\mathbf{B}_0\in \binom{\mathbf{C}}{\mathbf{B}}$ for which $c$ is constant on $\binom{\mathbf{B}_0}{\mathbf{A}}$. In [KPT], it is shown that for a Fra\"iss\'e structure $\mathbf{K}$, $\mathrm{Age}(\mathbf{K})$ has the RP and consists of rigid structures (having no non-trivial automorphisms) if and only if $M(\mathrm{Aut}(\mathbf{K}))$ consists of a single point; such groups are said to be \emph{extremely amenable}. It should be noted that we regard $\aut{\mathbf{K}}$ as a topological group with the pointwise convergence topology; a neighborhood basis at the identity is given by subgroups of the form $U_{\bar{a}} := \{g\in \aut{\mathbf{K}}: g(a_i) = a_i \text{ for each } a_i\in \bar{a}\}$ for finite $\bar{a}\subset \mathbf{K}$.

Let $\mathcal{K}$ and $\mathcal{K}^*$ be Fra\"iss\'e classes with limits $\mathbf{K}, \mathbf{K}^*$ in languages $L$ and $L^* = L\cup \{R_1,...,R_n\}$, with each $R_i$ a relation not in $L$, such that $\mathcal{K}^*|_L = \mathcal{K}$. We say the pair $(\mathcal{K}, \mathcal{K}^*)$ is \emph{reasonable} if for any $\mathbf{A}^*\in \mathcal{K}^*$, $\mathbf{B}\in \mathcal{K}$, and embedding $f: \mathbf{A}^*|_L\rightarrow \mathbf{B}$, there is $\mathbf{B}^*\in\mathcal{K}^*$ with $\mathbf{B}^*|_L = \mathbf{B}$ and $f: \mathbf{A}^* \rightarrow \mathbf{B}^*$ also an embedding. For $(\mathcal{K}, \mathcal{K}^*)$ reasonable, we have $\mathbf{K}^*|_L \cong \mathbf{K}$. In this case, $\mathrm{Aut}(\mathbf{K})$ acts on the compact, metrizable space $X_{\mathcal{K}^*}$ of all relations $(R_1,...,R_n)$ on $\mathbf{K}$ with $\mathrm{Age}(\langle \mathbf{K}, R_1,...,R_n\rangle) \subseteq \mathcal{K}^*$. The basic neighborhoods of $X_{\mathcal{K}^*}$ are given by open sets of the form $N_{\langle \mathbf{A}, S_1,...,S_n \rangle}:= \{(R_1,...,R_n)\in X_{\mathcal{K}^*}: R_i|_{\mathbf{A}} = S_i, 1\leq i\leq n\}$ for $\mathbf{A}\subseteq \mathbf{K}$, $\mathbf{A}\in \mathcal{K}$, and $\langle \mathbf{A}, S_1,...,S_n\rangle\in \mathcal{K}^*$. 

The pair satisfies the \emph{Expansion Property} if for each $\mathbf{A}\in \mathcal{K}$, there is $\mathbf{B}\in \mathcal{K}$ such that for any $\mathbf{A}^*, \mathbf{B}^*\in \mathcal{K}$ with $\mathbf{A}^*|_L= \mathbf{A}$, $\mathbf{B}^*|_L = \mathbf{B}$, we have $\mathbf{A}^* \leq \mathbf{B}^*$. If $(\mathcal{K},\mathcal{K}^*)$ is a reasonable pair, then the Expansion Property is equivalent to $X_{\mathcal{K}^*}$ being a minimal flow. Pairs of \fr classes which are reasonable, satisfy the Expansion Property, and with $\mathcal{K}^*$ satisfying the RP are called \emph{excellent}. In this case we call $\mathcal{K}^*$ a \textbf{companion} of $\mathcal{K}$. In [KPT], and later in the generality presented here in work of Nguyen Van Th\'e [LNVT], it is shown that given an excellent pair, $X_{\mathcal{K}^*}$ is the universal minimal flow of $\mathrm{Aut}(\mathbf{K})$. 

Let us show that for such $\mathbf{K}$, $\mathrm{Aut}(\mathbf{K})$ has the generic point property as witnessed by any point $(S_1,...,S_n)\in X_{\mathcal{K}^*}$ with $\langle \mathbf{K}, S_1,...,S_n\rangle \cong \mathbf{K}^*$. Notice that $\aut{\mathbf{K}}\cdot (S_1,...,S_n) = \{(R_1,...,R_n)\in X_{\mathcal{K}^*}: \langle \mathbf{K}, R_1,...,R_n\rangle \cong \mathbf{K}^*\}$; we need to show that this is $G_\delta$ (since $X_{\mathcal{K}^*}$ is minimal, we know that the orbit is dense). So fix $\mathbf{A}^*\subseteq \mathbf{B}^*\in \mathcal{K}^*$, with $\mathbf{A}^*|_L = \mathbf{A}$, $\mathbf{B}^*|_L = \mathbf{B}$ and let $f: \mathbf{A}\rightarrow \mathbf{K}$ be an embedding. Let $N_{f, \mathbf{A}^*\subseteq \mathbf{B}^*}$ denote those expansions $(R_1,...,R_n)\in X_{\mathcal{K}^*}$ for which $f$ embeds $\mathbf{A}^*$ into $\langle\mathbf{K}, R_1,...,R_n\rangle$ and can be extended to an embedding $g: \mathbf{B}^*\rightarrow \langle \mathbf{K}, R_1,...,R_n\rangle$. $N_{f, \mathbf{A}^*\subseteq \mathbf{B}^*}$ is open and represents a single instance of the Extension Property. Recall that structures isomorphic to $\mathbf{K}^*$ are exactly those which satisfy all instances of Extension Property; as there are countably many triples $f, \mathbf{A}^*, \mathbf{B}^*$, we see that $\aut{\mathbf{K}}\cdot (S_1,...,S_n)$ is $G_\delta$.

\section{Amenability}

Let $(\mathcal{K}, \mathcal{K}^*)$ be an excellent pair with limits $\mathbf{K}, \mathbf{K}^*$. Define $\mathrm{Fin}(\mathbf{K})$ to be the set of all finite substructures of $\mathbf{K}$. Compare this notion to $\mathrm{Age}(\mathbf{K}) = \mathcal{K}$; the latter consists of all structures isomorphic to structures in $\mathrm{Fin}(\mathbf{K})$. For each $\mathbf{A}\in \mathrm{Fin}(\mathbf{K})$, let
\begin{align*}
\mathcal{K}^*(\mathbf{A}) :=\{&\langle \mathbf{A},(S_1,...,S_n)\rangle\in \mathcal{K}^*\}
\end{align*}
Let $\mathbf{A}, \mathbf{B}\in \mathrm{Fin}(\mathbf{K})$, and let $\pi: \mathbf{A}\rightarrow \mathbf{B}$ be an embedding. For each $\mathbf{A}^* \in \mathcal{K}^*(\mathbf{A})$, define:
\begin{align*}
\mathcal{K}^*(\mathbf{A}^*, \mathbf{B}, \pi) := \{&\langle \mathbf{B}, (T_1,...,T_n)\rangle \in \mathcal{K}^*:\\ &\pi: \langle\mathbf{A}, S_1,...,S_n\rangle \rightarrow \langle\mathbf{B}, T_1,...,T_n\rangle \text{ is an embedding}\}.
\end{align*}

Let $\Omega = \!\!\!\displaystyle\bigcup_{\mathbf{A}\in \mathrm{Fin}(\mathbf{K})} \!\!\!\mathcal{K}^*(\mathbf{A})$. Though the elements of $\Omega$ are structures, for neatness we will often denote elements of $\Omega$ with variable names, such as $x$, $y$, etc. Form the vector space $\mathbb{Q}\Omega$, the vector space over $\mathbb{Q}$ with basis $\Omega$. Let $S$ be the set of all elements of the form
$$x - \left(\sum_{y \in \mathcal{K}^*(x, \mathbf{B}, \pi)}\!\!\!\!\!{y}\right)$$
for some $\mathbf{A}, \mathbf{B}\in \mathrm{Fin}(\mathbf{K})$, $x \in \mathcal{K}^*(\mathbf{A})$, and $\pi\!:\mathbf{A}\rightarrow \mathbf{B}$ an embedding. Let $V$ be the subspace generated by $S$. We are now able to state the first main theorem.

\begin{theorem}
$\mathrm{Aut}(\mathbf{K})$ is amenable if and only if for all nonzero $v\in V$, $v = \displaystyle\sum_{x \in \Omega}{c_x x}$, there are $x, y \in \Omega$ with $c_x > 0$, $c_y < 0$.
\end{theorem}
It should be noted that only the forward direction will be needed in sections 4 and 5.
\begin{proof}
$(\Rightarrow)$ Suppose $\mu$ is an invariant measure on $X_{\mathcal{K}^*}$. Recall that the topology on $X_{\mathcal{K}^*}$ is given by basic neighborhoods of the form $N_x = \{(S_1,...,S_n)^* \in X_{\mathcal{K}^*}: (S_1,...,S_n)^*|_{\mathbf{A}} = (S_1,...,S_n)\}$ for $x = \langle \mathbf{A}, (S_1,...,S_n)\rangle$. Define $\mu_{\Omega}: \mathbb{R}\Omega\rightarrow \mathbb{R}$ by 
$$\mu_{\Omega}\left(\displaystyle\sum_{x\in \Omega}{c_x x}\right) = \displaystyle\sum_{x\in \Omega}{c_x\mu (N_x)}.$$ 
We see that for $v\in V$, we must have $\mu_{\Omega}(v) = 0$. As $X_{\mathcal{K}^*}$ is minimal, $\mu_{\Omega}(x) = \mu (N_x) > 0$ for all $x\in \Omega$. If $0\neq v\in V$, $v = \sum_{x\in \Omega}{c_x x}$ and without loss of generality $c_x \geq 0$ for all $x\in \Omega$, then for some $x$ we have $c_x > 0$. But then $\mu_{\Omega}(v) > 0$, a contradiction.

$(\Leftarrow)$ Suppose the conditions of Theorem 3.1 hold. As noted by Kechris (private communication), it is sufficient to show that for $\mathbf{B}\in \mathrm{Fin}(\mathbf{K})$, there exists a consistent probability measure $\mu_\mathbf{B}$ on $\mathcal{K}^*(\mathbf{B})$, i.e.\ a probability measure such that for $\mathbf{A}\in \mathrm{Fin}(\mathbf{K})$, $x\in \mathcal{K}^*(\mathbf{A})$, and any two embeddings $\pi_1, \pi_2: \mathbf{A}\rightarrow \mathbf{B}$, we have $\mu_\mathbf{B}(\mathcal{K}^*(x, \mathbf{B}, \pi_1)) = \mu_\mathbf{B}(\mathcal{K}^*(x, \mathbf{B}, \pi_2))$. Indeed, if this is the case, let $\mathbf{A}_1\subset \mathbf{A}_2\subset...\subset \mathbf{K}$ be finite substructures with $\bigcup_{n=1}^{\infty} \mathbf{A}_n = \mathbf{K}$, and for each $n$ let $\mu_n$ be a consistent probability measure on $\mathcal{K}^*(\mathbf{A}_n)$. We will create an invariant measure on $X_{\mathcal{K}^*}$ as follows: let $\mathbf{A}\in \mathrm{Fin}(\mathbf{K})$ and $x\in \mathcal{K}^*(\mathbf{A})$; note that $\mathbf{A}\subset \mathbf{A}_n$ for large enough $n$. Set $\mu (N_x) = \displaystyle\lim_{n\to \mathcal{U}} \mu_n(\mathcal{K}^*(x, \mathbf{A}_n, i_{\mathbf{A}}))$, where $\mathcal{U}$ is a non-principal ultrafilter on $\mathbb{N}$ and $i_{\mathbf{A}}$ is the inclusion embedding. Extend to $X_{\mathcal{K}^*}$ by additivity. To see that this is invariant, let $g\in \aut{\mathbf{K}}$; find $\mathbf{B}\in \mathrm{Fin}(\mathbf{K})$ and $\pi:\; \mathbf{A}\rightarrow \mathbf{B}$ an embedding such that $\pi = g|_{\mathbf{A}}$. Let $n$ be sufficiently large so that $\mathbf{A}, \mathbf{B}\subset \mathbf{A}_n$. Now we have: 

\begin{align*}
\mu_n(\mathcal{K}^*(x, \mathbf{A}_n, i_{\mathbf{A}}))
 &= \mu_n(\mathcal{K}^*(x, \mathbf{A}_n, i_{\mathbf{B}}\circ\pi))\\
&= \mu_n\left(\displaystyle\bigsqcup_{y \in \mathcal{K}^*(x, \mathbf{B}, \pi)} \!\!\!\!\!\mathcal{K}^*(y, \mathbf{A}_n, i_\mathbf{B})\right)\\\\
\Rightarrow \mu(N_x) &= \lim_{n\to \mathcal{U}} \mu_n(\mathcal{K}^*(x, \mathbf{A}_n, i_{\mathbf{A}})\\ 
&= \lim_{n\to\mathcal{U}} \!\sum_{y\in \mathcal{K}^*(x,\mathbf{B},\pi)} \!\!\!\mu_n(\mathcal{K}^*(y, \mathbf{A}_n, i_{\mathbf{B}})\\
&= \sum_{y\in \mathcal{K}^*(x, \mathbf{B}, \pi)} \mu(N_y)\\
&= \mu(g(N_x)).
\end{align*}  

Let $S_\mathbf{B}\subset V$ consist of all elements of $\mathbb{R}\Omega$ of the form
\vskip 2 mm
\begin{align*}
\left(\sum_{y\in \mathcal{K}^*(x, \mathbf{B}, \pi_1)}\!\!\!\!\!{y}\right)- \left(\sum_{z\in \mathcal{K}^*(x, \mathbf{B}, \pi_2)}\!\!\!\!\!{z}\right) = \left(x-\left(\sum_{z\in \mathcal{K}^*(x, \mathbf{B}, \pi_2)}\!\!\!\!\!{z}\right)\right) - \left(x-\left(\sum_{y\in \mathcal{K}^*(x, \mathbf{B}, \pi_1)}\!\!\!\!\!{y}\right)\right)\\
\end{align*}
for some $\mathbf{A}\in \mathrm{Fin}(\mathbf{K})$, $x\in \mathcal{K}^*(\mathbf{A})$, and embeddings $\pi_1,\pi_2:\; \mathbf{A}\rightarrow \mathbf{B}$. Consider the following system of inequalities and equalities in real variables $q_x$, $x\in \mathcal{K}^*(\mathbf{B})$, where for $v = \displaystyle\sum_{x\in \Omega} c_x x\in \mathbb{R}\Omega$, we let $q_v = \displaystyle\sum_{x\in \Omega} c_x q_x$.
\vskip -1 mm
\begin{alignat}{2}
q_s &= 0 &\quad 
&(s\in S_\mathbf{B}),\\
q_x &> 0 &&(x\in \mathcal{K}^*(\mathbf{B})),\\
\medskip
\sum_{x\in \mathcal{K}^*(\mathbf{B})}{q_x} &= 1. & &
\end{alignat}

If this system has a solution, the solution is the consistent probability measure we seek, with $\mu_{\mathbf{B}}(x) = q_x$. Note that the system (1), (2), (3) has a solution when the system (1), (2) has a solution. Now we can use the rational form of Stiemke's theorem ([S], see also Border [B]), which states that for a rational-valued matrix $A$, the equation $A\mathbf{x} = 0$ has a rational solution with $\mathbf{x}>0$ {\rm (}each entry $x\in \mathbf{x}$ has $x>0${\rm )}, or the equation $\mathbf{y}^TA \gneq 0$ has a rational solution {\rm (}each entry $z\in \mathbf{y}^TA$ has $z\geq 0$, and at least one entry has $z>0${\rm )}.

Let $A$ be the matrix of coefficients from (1). A solution to the system (1), (2) is a vector $\mathbf{x}$ whose entries are exactly the $q_x$. Hence, if there is no solution to $A\mathbf{x}>0$, let $\mathbf{y} = \{y_s: s\in S_\mathbf{B}\}$ be a solution to $\mathbf{y}^TA\gneq 0$. Consider the element $v = \displaystyle\sum_{s\in S_{\mathbf{B}}} y_s s \in V$. Write $v = \displaystyle\sum_{x\in \mathcal{K}^*(\mathbf{B})} c_x x$. Stiemke's theorem tells us that $c_x \geq 0$ for all $x\in \mathcal{K}^*(\mathbf{B})$ and that for some $x$ we have $c_x > 0$. This contradicts the assumptions of Theorem 3.1. 
\end{proof}

\section{Application: $\mathrm{Aut}(\mathbf{S}(3))$ is not amenable}
The directed graph $\mathbf{S}(3)$ has vertex set $\{e^{iq}: q\in \mathbb{Q}\}$; for vertices $a = e^{iq_1}$, $b = e^{iq_2}$, $a\rightarrow b$ when there is $m\in \mathbb{Z}$ with $q_2 - q_1 \in (2\pi m, 2\pi m + 2\pi/3)$, i.e.\ $b$ is less than $2\pi /3$ radians counterclockwise from $a$. We will write $(a, b)$ for those points lying in the interval counterclockwise from $a$ to $b$.  Let $\mathcal{K} = \mathrm{Age}(\mathbf{S}(3))$. The appropriate companion class with the Expansion and Ramsey properties is $\mathcal{K}^* = \mathrm{Age}(\mathbf{S}(3)^*)$, where $\mathbf{S}(3)^* = \langle \mathbf{S}(3), P_0, P_1, P_2 \rangle$, and the $P_i$ are unary relations with $P_i(a)$ when $a$ is between $2\pi i/3$ and $2\pi (i+1)/3$ radians counterclockwise from the top of the circle (see [LNVT]). For $a\in \mathbf{A}^*\in \mathcal{K}^*$, we will write $P(a) = i$ when $P_i(a)$ holds. 

\begin{center}
\includegraphics*[viewport=0 250 500 700, scale=0.3]{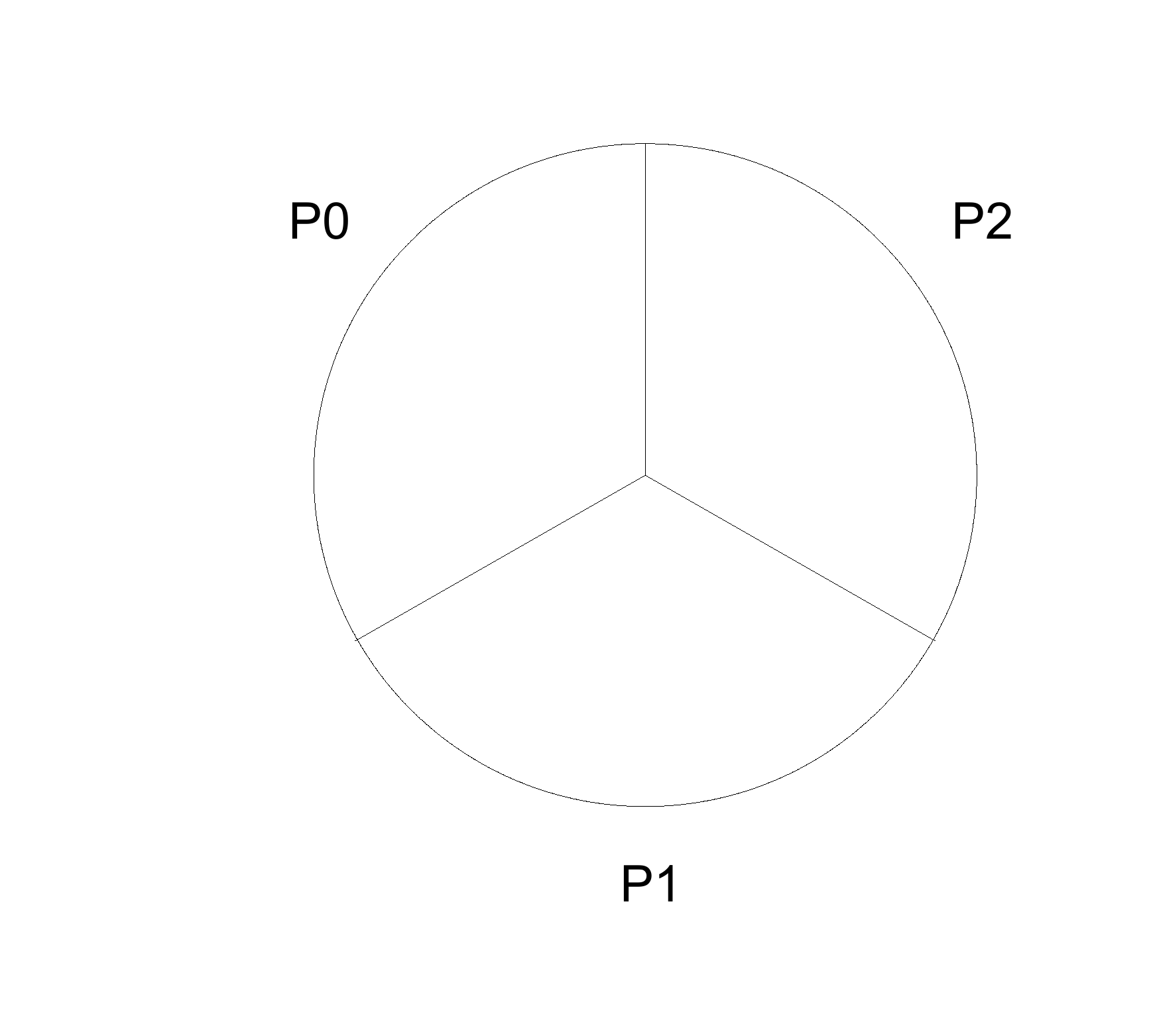}
\end{center}

We will use Theorem 3.1 to show:
\begin{theorem}
$\mathrm{Aut}(\mathbf{S}(3))$ is not amenable.
\end{theorem}
\begin{proof}
Let $\mathbf{A}\in \mathrm{Fin}(\mathbf{K})$ be a directed graph consisting of two vertices $a$ and $b$ with no edge between them. Let $\mathbf{B}\in \mathrm{Fin}(\mathbf{K})$ be a directed graph with vertices $w$, $x$, $y$, and $z$, with edges $w\rightarrow x$, $x\rightarrow y$, and $y\rightarrow z$.
 
\begin{center}
\includegraphics*[viewport=50 450 300 750, scale=0.5]{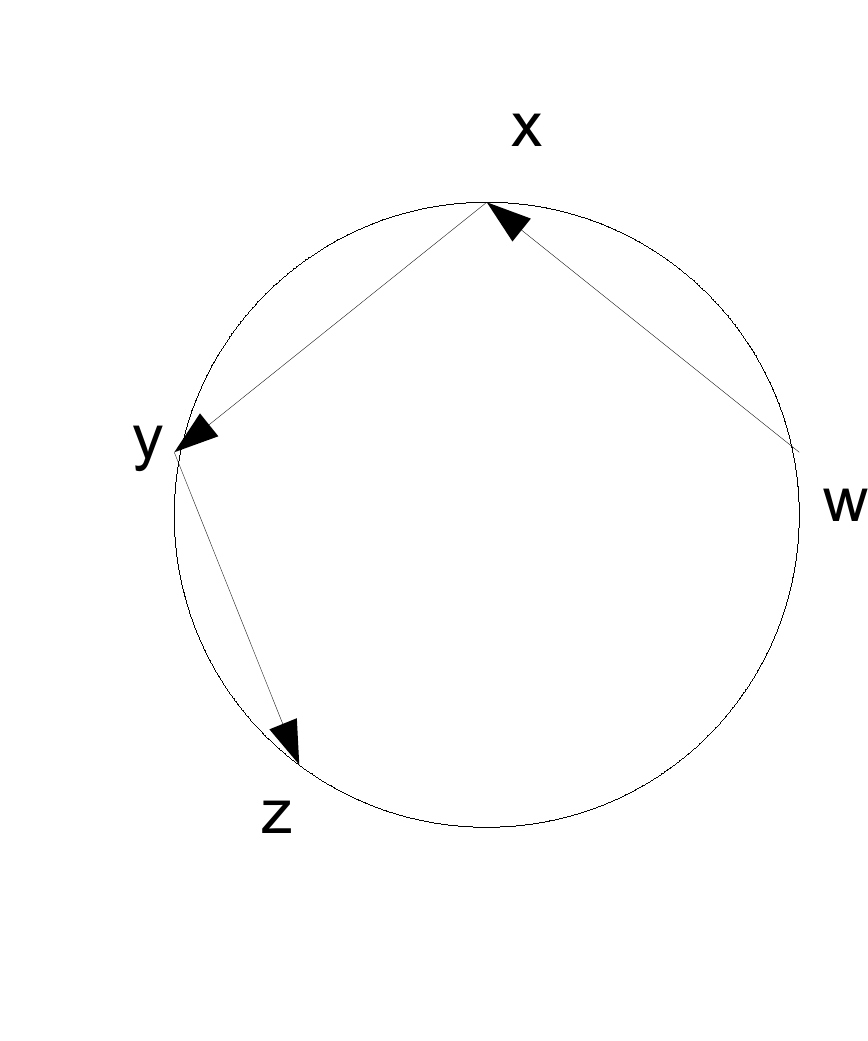}
\end{center}

Consider $\mathcal{K}^*(\mathbf{B})$; we list all 12 expansions of $\mathbf{B}$ below, where $(i, j, k, l)$ stands for $\langle \mathbf{B}, (P_1, P_2, P_3)\rangle \in \mathcal{K}^*(\mathbf{B})$ with $P_i(w)$, $P_j(x)$, $P_k(y)$, $P_l(z)$.
\begin{align*}
\mathcal{K}^*(\mathbf{B}) = \{&(0, 1, 1, 2), (0, 1, 2, 2), (1, 1, 2, 2), (1, 1, 2, 0),\\ 
&(1, 2, 2, 0), (1, 2, 0, 0), (2, 2, 0, 0), (2, 2, 0, 1),\\ 
&(2, 0, 0, 1), (2, 0, 1, 1), (0, 0, 1, 1), (0, 0, 1, 2)\}.
\end{align*}
To arrive at the above, first observe that we cannot have $P(w)=P(x)=P(y)$ nor $P(x)=P(y)=P(z)$, as this would require $w\rightarrow y$ or $x\rightarrow z$, respectively. Also observe that mod $3$, we cannot have any of $P(w) = P(x)+1$, $P(x) = P(y)+1$, nor $P(y) = P(z)+1$; this is because for any $a, b\in \mathbf{S}(3)^*$, we have $a\rightarrow b \Rightarrow P(a)\neq P(b)+1$. Lastly, we may not have $P(w) = P(z)$, as this would imply $z\rightarrow w$ or $w \rightarrow z$. Now without loss of generality fix $w \in P_1$. There are exactly four possible expansions meeting the necessary conditions, and all four are easily realized. The other eight expansions are then given by adding $1$ or $2$ mod 3 to each coordinate 
   
Let $x\in \mathcal{K}^*(\mathbf{A})$ denote the expansion with $P_0(a)$, $P_1(b)$. Let $\pi_1: \mathbf{A}\rightarrow \mathbf{B}$ denote the embedding $\pi_1(a) = w$, $\pi_1(b) = y$. Let $\pi_2: \mathbf{A}\rightarrow \mathbf{B}$ denote the embedding $\pi_2(a) = w$, $\pi_2(b) = z$. Now we have
\begin{align*}
\mathcal{K}^*(x, \mathbf{B}, \pi_1) &= \{(0, 1, 1, 2), (0, 0, 1, 1), (0, 0, 1, 2)\},\\
\mathcal{K}^*(x, \mathbf{B}, \pi_2) &= \{(0, 0, 1, 1)\}.
\end{align*}
Forming the spaces $\mathbb{R}\Omega$, $V$ as above, we have
\begin{align*}
[(0, 1, 1, 2) + (0, 0, 1, 1) + (0, 0, 1, 2)] - [(0, 0, 1, 1)] &\in V,\\
\Rightarrow (0, 1, 1, 2) + (0, 0, 1, 2) &\in V.
\end{align*}
Hence $\mathrm{Aut}(\mathbf{S}(3))$ is not amenable.
\end{proof}

\section{Application: Boron Trees}
A boron tree is a graph-theoretic unrooted tree where each vertex has degree $1$ or $3$. It is possible to interpret boron trees as a model theoretic structure such that the class of finite boron trees forms a Fra\"iss\'e class; we will closely follow the exposition of Jasi\'nski [J]. For a boron tree $T$, let $(T) = \langle L_T, R\rangle$ be the structure with universe $L_T$, the leaves of $T$, and $R$ a $4$-ary relation defined as follows: for leaves $a$, $b$, $c$, $d$, we have $R(a, b, c, d)$ if $a$, $b$, $c$, $d$ are distinct and there are paths from $a$ to $b$ and from $c$ to $d$ which do not intersect. It is a fact that for finite boron trees $T$ and $U$, $T\cong U$ if and only if $(T)\cong (U)$.

Let $\mathbf{B}(n)$ be the boron tree structure defined as follows: its universe is $B_n = 2^n = \{f: n\rightarrow 2\}$. For any two leaves $a$ and $b$, let $\delta(a,b) = \mathrm{max}(k:\: a|_k = b|_k)$. Define $[a, b] = \{a|_k:\; k\in [\delta(a,b), n]\}\cup \{b|_k:\; k\in [\delta(a,b), n]\}$. For distinct $a$, $b$, $c$, $d$, set $R(a,b,c,d)$ if $[a,b]\cap [c,d] = \emptyset$. The associated tree has vertices $2^{\leq n}\backslash \emptyset$; vertices $u:\; k\rightarrow 2$ and $v:\; k+1\rightarrow 2$ are adjacent if $v|_k = u$, and additionally the vertices $x_0, x_1: 1\rightarrow 2$ with $x_0(0) = 0$, $x_1(0) = 1$ are adjacent.

\begin{center}
\includegraphics*[viewport=50 380 450 700, scale=0.4]{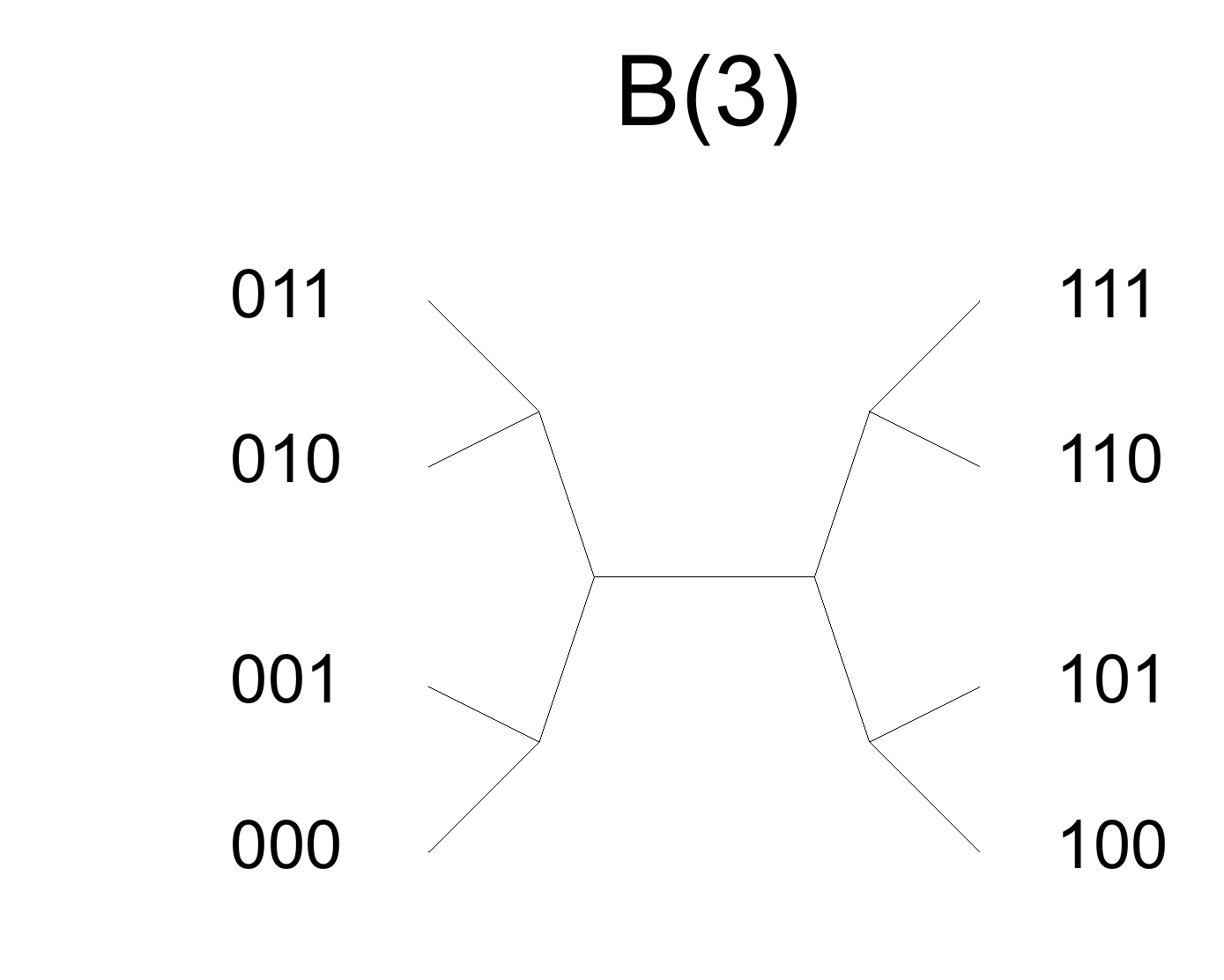} 
\end{center}

On each $\mathbf{B}(n)$, let $<_l$ be the lexicographic linear ordering of the leaves; $a <_l b$ if $a(0) < b(0)$ or $a(0) = b(0)$ and $a(1) < b(1)$ or etc. Also, for any function $f:\; k\rightarrow 2$, $k\leq n$, let $B_f(n)$ consist of those $a\in B(n)$ with $a|_k = f$.  

Let $\mathcal{B}$ be the class of finite boron tree structures. It is a fact that the set $\{\mathbf{B}(n):\; n\in\mathbb{N}\}$ is cofinal in $\mathcal{B}$; for each $\mathbf{A}\in \mathcal{B}$, there is an $n$ for which there is an embedding $\pi:\; \mathbf{A} \rightarrow \mathbf{B}(n)$. We will use this fact in defining the companion class $\mathcal{B}^*$ such that $(\mathcal{B}, \mathcal{B}^*)$ is an excellent pair. Let $\mathbf{A}\in \mathcal{B}$, and let $\pi:\; \mathbf{A} \rightarrow \mathbf{B}(n)$ be an embedding for some $n$. Define the structure $o(\mathbf{A}, \pi) = \langle \mathbf{A}, S\rangle$ as follows: $S$ is a $3$-ary relation, where for $a, b, c\in A$, we have $S(a,b,c)$ if $\pi(a),\pi(b) <_l \pi(c)$ and $\delta(\pi(a),\pi(b)) > \delta(\pi(b),\pi(c))$. We can rephrase this condition in a somewhat easier to grasp way: for $x,y\in B(n)$, let $M(x,y)\in B(n)$ be as follows: 

\begin{equation*}
M(x,y)(j) = 
\begin{cases}
x(j) & \text{if $j < \delta(x,y)$},\\
1 & \text{if $j\geq \delta(x,y)$}.
\end{cases}
\end{equation*}
\vspace{2 mm}

Now we have $S(a,b,c)$ exactly when $M(\pi(a),\pi(b)) <_l \pi(c)$. Note that $S(a,b,c)$ holds whenever $S(b,a,c)$ does. Let $\mathcal{B}^*$ be the class of all such $o(\mathbf{A}, \pi)$. It is a fact that $(\mathcal{B},\mathcal{B}^*)$ is an excellent pair (see [J]), with limits $\mathbf{T}, \mathbf{T}^*$. We will show:
\begin{theorem}
$\mathrm{Aut}(\mathbf{T})$ is not amenable.
\end{theorem}
\begin{proof}
We begin with the following proposition.
\begin{prop}
Let $o(\mathbf{A}, \pi)\in \mathcal{B}^*$, with $\pi:\; \mathbf{A} \rightarrow \mathbf{B}(n)$, and $|A| = l$. Then there exists an embedding $\phi:\; \mathbf{A} \rightarrow \mathbf{B}(l-1)$ with $o(\mathbf{A}, \pi) = o(\mathbf{A}, \phi)$.
\end{prop}

\begin{proof}
Suppose there exists $k\in \{0,1,...,n-1\}$ such that for any $a, b\in A$, we have $\pi(a)|_k = \pi(b)|_k \Rightarrow \pi(a)|_{k+1} = \pi(b)|_{k+1}$, or equivalently that for no $a, b\in A$ do we have $\delta(\pi(a),\pi(b))=k$. We will show that there is an embedding $\psi:\; \mathbf{A} \rightarrow \mathbf{B}(n-1)$ with $o(\mathbf{A}, \psi) = o(\mathbf{A}, \pi)$. For $a\in B(n)$, define the map $f_k:\; B(n) \rightarrow B(n-1)$ as follows.
\vspace{2 mm}
\begin{equation*}
f_k(a)(j) = 
\begin{cases}
a(j) & \text{if $j<k$},\\
a(j+1) & \text{if $j\geq k$}.
\end{cases}
\end{equation*}
\vspace{2 mm}

Define $\psi:\; A\rightarrow B(n-1)$ by $\psi = f_k\circ \pi$. First let us show that $\psi$ is an embedding $\mathbf{A}\rightarrow \mathbf{B}(n-1)$. It will be useful to note that
\begin{equation*}
\delta(\psi(a),\psi(b)) = 
\begin{cases}
\delta(\pi(a),\pi(b)) & \text{if $\delta(\pi(a),\pi(b))<k$},\\
\delta(\pi(a),\pi(b))-1 & \text{if $\delta(\pi(a),\pi(b))>k$}.
\end{cases}
\end{equation*}

Let $a,b,c,d\in A$ be distinct, and suppose $R(a,b,c,d)$. As $\pi$ is an embedding, we have $[\pi(a),\pi(b)] \cap [\pi(c),\pi(d)]= \emptyset$. Suppose, without loss of generality and for sake of contradiction, that $\psi(a)|_m = \psi(c)|_m \in [\psi(a), \psi(b)] \cap [\psi(c), \psi(d)]$. If $m<k$, we have that $\pi(a)|_m = \pi(c)|_m\in [\pi(a),\pi(b)] \cap [\pi(c),\pi(d)]$, a contradiction. If $m\geq k$, then $\pi(a)|_{m+1} \in [\pi(a),\pi(b)]$ and $\pi(c)|_{m+1} \in [\pi(c),\pi(d)]$. Therefore we must have $\pi(a)|_{m+1} \neq \pi(c)|_{m+1}$. However, this is only possible if $\pi(a)(k) \neq \pi(c)(k)$. As $\pi(a)|_k = \pi(c)|_k$, this is a contradiction. Therefore we have $R(\psi(a), \psi(b), \psi(c), \psi(d))$.
\vspace{2 mm}

If $\neg R(a,b,c,d)$, we may suppose $\pi(a)|_m = \pi(c)|_m \in [\pi(a),\pi(b)] \cap [\pi(c),\pi(d)]$. If $m<k$, then $\psi(a)|_m = \psi(c)|_m \in [\psi(a), \psi(b)] \cap [\psi(c), \psi(d)]$. If $m \geq k$, then $\psi(a)|_{m-1} = \psi(c)|_{m-1} \in [\psi(a), \psi(b)] \cap [\psi(c), \psi(d)]$. Hence we have $\neg R(\psi(a), \psi(b), \psi(c), \psi(d))$.
\vspace{2 mm}

Now we will show that $o(\mathbf{A}, \psi) = o(\mathbf{A}, \pi)$. Suppose $S^{\pi}(a,b,c)$, or equivalently $M(\pi(a), \pi(b)) <_l \pi(c)$. Observe that $M(\psi(a), \psi(b)) = f_k(M(\pi(a), \pi(b)))$. Now if $\psi(c) <_l M(\psi(a), \psi(a))$, we must have $\pi(c)|_k = M(\pi(a),\pi(b))|_k$; it follows that $\delta(\pi(a),\pi(b)) > k$. Let $m$ be least such that $\pi(c)(m) \neq \pi(a)(m)$. Then we must have $k < m < \delta(\pi(a),\pi(b))$. Then $m-1$ is least such that $\psi(c) \neq \psi(a)$, and $m-1 < \delta(\psi(a),\psi(c))$. Hence $S^{\psi}(a,b,c)$. 
\vspace{0 mm}

Now suppose $S^{\psi}(a,b,c)$. Let $m$ be least such that $\psi(c)(m) \neq \psi(a)(m)$. Then $m < \delta(\psi(a),\psi(b)) \leq \delta(\pi(a),\pi(b))$. Now if $m < k$, then $m$ is least such that $\pi(c)(m) \neq \pi(a)(m)$. If $m > k$, then $m+1$ is least such that $\pi(c)(m+1) \neq \pi(a)(m+1)$, and $m+1 < \delta(\pi(a), \pi(b))$. Hence $S^{\pi}(a,b,c)$, and $o(\mathbf{A}, \psi) = o(\mathbf{A}, \pi)$.
\vspace{2 mm}

Now suppose no such $k$ exists. We will show that $|A| \geq n+1$. More precisely, we will show that for any $m$ and embedding $\phi:\; \mathbf{A}\rightarrow \mathbf{B}(m)$, the map $\delta_{\phi}:\; A\times A\rightarrow m+1$ with $\delta_{\phi} (a,b) = \delta (\phi(a),\phi(b))$ has $|\delta_{\phi}(A\times A)|\leq n$. For $n=1$ this is clear. Assume the result true for $n=l$, and suppose $|A| = l+1$. Fix $a\in A$, and let $b\neq a$ be such that $\delta_{\phi} (a,b)$ is maximal. Let $c\in A$ with $c\neq a$ and $c\neq b$. Then
\begin{equation*}
\delta_{\phi} (a,c) = 
\begin{cases}
\delta_{\phi} (a,b) & \text{if $\delta_{\phi} (b,c) >\delta_{\phi} (a,b)$},\\
\delta_{\phi} (b,c) & \text{if $\delta_{\phi} (b,c) <\delta_{\phi} (a,b)$}.
\end{cases}
\end{equation*} 
We see that $|\delta_{\phi}(A\times A)| = |\delta_{\phi} (A\backslash a \times A\backslash a)\cup \{\delta_{\phi} (a,b)\}| \leq l+1$
\end{proof} 

With this result, we can now compute $\mathcal{B}^*(\mathbf{A})$ for $\mathbf{A}\in \mathrm{Fin}(\mathbf{T})$; we will exhibit an element of $V\subset \mathbb{R}\Omega$ showing that $\mathrm{Aut}(\mathbf{T})$ is not amenable. Let $\mathbf{A}\in \mathrm{Fin}(\mathbf{T})$ be a boron tree with three leaves $a,b,c$. We list the $12$ expansions in $\mathcal{B}^*(\mathbf{A})$ below; we will denote an element of $\mathcal{B}^*(\mathbf{A})$ by $[a_1<...<a_k;\; (x_1, y_1, z_1),...,(x_l, y_l, z_l)]$, where we have $S(a_i, a_i, a_j)$ for $i<j$, $S(x_i, y_i, z_i)$ and $S(y_i, x_i, z_i)$ for $1\leq i\leq l$.
\vskip -1 mm
\begin{align*}
\mathcal{B}^*(\mathbf{A}) =  \{ & [a<b<c;\; \emptyset],[a<c<b;\; \emptyset],[b<a<c;\; \emptyset],\\
& [b<c<a;\; \emptyset],[c<a<b;\; \emptyset],[c<b<a;\; \emptyset],\\
& [a<b<c;\; (a,b,c)],[a<c<b;\; (a,c,b)],[b<a<c;\; (b,a,c)],\\
& [b<c<a;\; (b,c,a)],[c<a<b;\; (c,a,b)],[c<b<a;\; (c,b,a)]\}.
\end{align*}

Now consider $\mathbf{B}(2)\in \mathrm{Fin}(\mathbf{T})$, where $B(2) = \{w = 00, x = 01, y = 10, z = 11\}$; this structure has $40$ expansions. These can be split into $5$ types, each of which has $8$ expansions.

\begin{enumerate}
\item
Type $A$: Expansions of the form $o(\mathbf{B}(2), \pi)$ with $\pi:\; \mathbf{B}(2)\rightarrow \mathbf{B}(2)$. Below, $(s,t,u,v)$ stands for the expansion $[s<t<u<v; (s,t,u), (s,t,v)]$.
\begin{align*}
A = \{a_1,...,a_8\} = \{&(w,x,y,z), (w,x,z,y), (x,w,y,z), (x,w,z,y),\\
& (y,z,w,x), (y,z,x,w), (z,y,w,x), (z,y,x,w)\}.
\end{align*}

\item
Type $B$: Expansions of the form $o(\mathbf{B}(2), \pi)$ with $\pi:\; \mathbf{B}(2)\rightarrow \mathbf{B}(3)$ such that $|B_0(3)\bigcap \pi(A)| = 1$ and $|B_{10}(3)\bigcap \pi(A)| = 1$. Below, $(s,t,u,v)$ stands for the expansion $[s<t<u<v; \emptyset]$.
\begin{align*}
B = \{b_1,...,b_8\} = \{&(w,x,y,z), (w,x,z,y), (x,w,y,z), (x,w,z,y),\\
& (y,z,w,x), (y,z,x,w), (z,y,w,x), (z,y,x,w)\}.
\end{align*}

\item
Type $C$: Expansions of the form $o(\mathbf{B}(2), \pi)$ with $\pi:\; \mathbf{B}(2)\rightarrow \mathbf{B}(3)$ such that $|B_0(3)\bigcap \pi(A)| = 1$ and $|B_{11}(3)\bigcap \pi(A)| = 1$. Below, $(s,t,u,v)$ stands for the expansion $[s<t<u<v; (t,u,v)]$.
\begin{align*}
C = \{c_1,...,c_8\} = \{&(w,y,z,x), (w,z,y,x), (x,y,z,w)), (x,z,y,w),\\
& (y,w,x,z), (y,x,w,z), (z,w,x,y), (z,x,w,y)\}.
\end{align*}

\item
Type $D$: Expansions of the form $o(\mathbf{B}(2), \pi)$ with $\pi:\; \mathbf{B}(2)\rightarrow \mathbf{B}(3)$ such that $|B_1(3)\bigcap \pi(A)| = 1$ and $|B_{00}(3)\bigcap \pi(A)| = 1$. Below, $(s,t,u,v)$ stands for the expansion $[s<t<u<v; (s,t,v), (s,u,v), (t,u,v)]$.
\begin{align*}
D = \{d_1,...,d_8\} = \{&(w,y,z,x), (w,z,y,x), (x,y,z,w)), (x,z,y,w),\\
& (y,w,x,z), (y,x,w,z), (z,w,x,y), (z,x,w,y)\}.
\end{align*}

\item
Type $E$: Expansions of the form $o(\mathbf{B}(2), \pi)$ with $\pi:\; \mathbf{B}(2)\rightarrow \mathbf{B}(3)$ such that $|B_1(3)\bigcap \pi(A)| = 1$ and $|B_{01}(3)\bigcap \pi(A)| = 1$. Below, $(s,t,u,v)$ stands for the expansion $[s<t<u<v; (s,t,u), (s,t,v), (s,u,v), (t,u,v)]$.
\begin{align*}
E = \{e_1,...,e_8\} = \{&(w,x,y,z), (w,x,z,y), (x,w,y,z), (x,w,z,y),\\
& (y,z,w,x), (y,z,x,w), (z,y,w,x), (z,y,x,w)\}.
\end{align*}
\end{enumerate}

Now set $x = [a<b<c; (a,b,c)]\in \mathcal{B}^*(\mathbf{A})$. Let $\pi_1,\pi_2:\; \mathbf{A}\rightarrow \mathbf{B}(2)$ with $\pi_1(a,b,c) = (w,x,y)$, $\pi_2(a,b,c) = (w,y,z)$. We see that $\mathcal{B}^*(x,\mathbf{B}(2),\pi_1) = \{a_1, a_2, c_7, d_7, e_1, e_2\}$ and $\mathcal{B}^*(x,\mathbf{B}(2),\pi_2) = \{e_1, e_3\}$. Thus in $\mathbb{R}\Omega$,
\begin{equation*}
a_1+a_2+c_7+d_7+e_2-e_3\in V.
\end{equation*}
Now let $\phi_1,\phi_2: \mathbf{B}(2)\rightarrow \mathbf{B}(2)$ with $\phi_1(w,x,y,z) = (w,x,y,z)$, $\phi_2(w,x,y,z) = (x,w,z,y)$. We have that $\mathcal{B}^*(e_2, \mathbf{B}(2), \phi_1) = \{e_2\}$, and $\mathcal{B}^*(e_2, \mathbf{B}(2), \phi_2) = \{e_3\}$. In $\mathbb{R}\Omega$,
\begin{align*}
e_2-e_3 &\in V,\\
\Rightarrow a_1+a_2+c_7+d_7 &\in V.
\end{align*}
Therefore $\mathrm{Aut}(\mathbf{T})$ is not amenable.
\end{proof}

\section{Proof of Theorem 1.2}

For the remainder of this paper, we shift our focus to the class $\mathcal{K}$, the class of finite dimensional vector spaces over a fixed finite field $F_q$, which has companion $\mathcal{K}^*$, the class of such vector spaces equipped with a natural linear ordering, an ordering induced antilexicographically by some choice of ordered basis and a fixed ordering of $F_q$ with $0<1$ the least elements. This is to say that if $b_0>\cdots >b_{n-1}$ is the ordered basis we have chosen and $x_i, y_i\in F_q$, then $x_0b_0+\cdots +x_{n-1}b_{n-1}> y_0b_0+\cdots +y_{n-1}b_{n-1}$ if $x_\ell > y_\ell$, where $\ell$ is the least index where $x_\ell \neq y_\ell$. In particular, notice that there is a 1-1 correspondence between ordered bases and natural orderings. In what follows, all vector spaces are assumed to be over $F_q$. Set $\mathrm{Flim}(\mathcal{K}) = \mathbf{V}_\infty$, the countably infinite dimensional vector space over $F_q$, and $\mathrm{Flim}(\mathcal{K}^*) = \mathbf{V}_\infty^*$. It is shown in [AKL] that $\aut{\mathbf{V}_\infty}$ is uniquely ergodic. The unique measure is just the uniform measure: for any finite $V\subseteq \mathbf{V}_\infty$, there are $|\mathrm{GL}(V)|$ admissible expansions, each of which has measure $1/|\mathrm{GL}(V)|$. Another way to see this is to notice the following: if $U$ is finite dimensional, $V, W\subseteq U$ have the same dimension, and $<_V, <_W$ are natural linear orderings on $V, W$, then there is an automorphism of $U$ sending $\langle V, <_V\rangle$ to $\langle W, <_W\rangle$.

Given any vector space $V$ with any ordering $<$, say that $\langle V, <\rangle$ has the \emph{Finite Lex Property} (FLP) if for any finite dimensional subspace $U\subset V$, we have that $<\!|_U$ is a natural ordering. In particular, a finite dimensional, ordered vector space has the FLP iff that ordering is natural (See Thomas [Th]), and $\langle \mathbf{V}_{\infty}, <\rangle$ has the FLP iff $< \,\in \!X_{\mathcal{K}^*}$. Given $v\in V$, $v \neq 0$, say that $v$ is \emph{minimal in its line} if $v\leq cv$ for all $c\in F_q\backslash \{0\}$. Given $u, v\in V$, define the relation $u \ll v \Leftrightarrow cu < dv$ for all $c, d\in F_q\backslash \{0\}$. Also define $u\sim v \Leftrightarrow (u \not\ll v) \wedge (v\not\ll u)$. For $q=2$, we need to tweak these definitions a bit. We define $u\sim v$ to hold when $u+v < min(u,v)$. We have $u\ll v$ if $u < v$ and $u\nsim v$. Note that if $u\sim v$, then $cu\sim dv$ for all $c,d\in F_q\backslash \{0\}$. Also, if $u,v\in \langle V, <\!|_V\rangle\subset\langle U, <\rangle$, then $u \ll_V v \Leftrightarrow u \ll_U v$ and $u\sim_V v\Leftrightarrow u\sim_U v$. 

When $\langle V, <\rangle$ is finite dimensional, $\sim$ has a simple characterization. Say $<$ is given by ordered basis $Z = \{z_0>\cdots >z_{n-1}\}$. For $v\in V\backslash \{0\}$, write $v = v_0z_0+\cdots +v_{n-1}z_{n-1}$, and suppose $l$ is least with $v_l$ nonzero. Then for $u = u_0z_0+\cdots +u_{n-1}z_{n-1}$, we have $u\sim v$ exactly when $l$ is also least with $u_l$ nonzero.

\begin{lemma}
For $\langle V, <\rangle$ with the FLP, $\sim$ is an equivalence relation.
\end{lemma}
\begin{proof}
Suppose $u\sim v$ and $v\sim w$. Form $\langle u, v, w\rangle := U$, and let $\{x_0>...>x_l\}$ be the basis of $U$ inducing $<\!|_U$, $l\leq 2$. Write $u = u_0x_0 +...+u_lx_l$, $u_i\in F_q$; likewise for $v$, $w$. Observe that one of $u_0$, $v_0$, $w_0$ is nonzero; as $u\sim v$ and $v\sim w$, all three must be nonzero. It follows that $u\sim w$.
\end{proof}

Observe that $V/\!\sim$, the set of nonzero equivalence classes, is also linearly ordered by $A<B \Leftrightarrow a<b$ for all $a\in A$, $b\in B$. We will denote by $<\!\!/\!\!\sim$ the order type of this linear ordering. It will be useful to introduce a standard notation for the $n$ nonzero equivalence classes of an $n$-dimensional, naturally ordered vector space $\langle V, <\rangle$. Call these classes $[V,<]_{n-1} <...<[V,<]_0$; note that $v\in [V,<]_l$ exactly when $l$ is least with $v_l$ nonzero.

Fix $V_1\!\subset \!\mathbf{V}_\infty$ a $1$-dimensional subspace. Given $<\in X_{\mathcal{K}^*}$, observe that for any $u, v\in V_1\backslash \{0\}$ we have $u\sim v$. Call this equivalence class $[V_1, <]$. Define
$$N_{V_1}^k = \{<\,\in \!X_{\mathcal{K}^*}:\; [V_1, <] < [u] \text{ for at most } k \text{ equivalence classes } [u]\in \mathbf{V}_\infty /\!\sim\}.$$
Note that $N_{V_1}^k \subseteq N_{V_1}^l$ when $k\leq l$. Let $N_{V_1}^{fin} = \bigcup_{k\in \mathbb{N}} N_{V_1}^k$. 

\begin{prop}
For any $<\in N_{V_1}^{fin}$, $\langle \mathbf{V}_\infty, <\rangle$ is not a \fr structure.
\end{prop}

\begin{proof}
We will show that the extension property does not hold for $\langle \mathbf{V}_\infty, <\rangle$. To see this, let $\langle V, \prec\!|_V\rangle \subset \langle U, \prec\rangle \in \mathcal{K}^*$, where $U$ is $k+2$-dimensional, $\prec$ on $U$ is given by an ordered basis $u_0 \succ...\succ u_{k+1}$, and $V = \langle u_{k+1}\rangle$. Now given $<\in X_{\mathcal{K}^*}$, let $\pi:\: \langle V, \prec\!|_V\rangle\rightarrow \langle \mathbf{V}_\infty, <\rangle$ with $\pi(V) = V_1$. There are at most $k$ equivalence classes greater than $[V_1, <]$ in $\mathbf{V}_\infty$, but $k+1$ equivalence classes greater than $[V, \prec] = [U,\prec]_{k+1}$ in $U$. It follows that there is no $\pi':\; \langle U, \prec\rangle \rightarrow \langle \mathbf{V}_\infty, <\rangle$ extending $\pi$. 
\end{proof}

Somewhat conversely, if there are at least $k+1$ equivalence classes above $[V_1, <]$ in $\langle \mathbf{V}_\infty, <\rangle$, then $\pi$ as in the proof of 6.2 does admit an extension $\pi'$. We can see this as follows: choose representatives $w_i$, $0\leq i \leq k$ from $k+1$ equivalence classes $[w_0]>\cdots >[w_k]$, where $[w_k]>[V_1, <]$. Let $W = \langle w_0,...,w_k, V_1\rangle\subset \mathbf{V}_\infty$. Then $W$ is a $k+2$ dimensional space; let $\{x_0>\cdots >x_{k+1}\}$ be the basis which gives the ordering $<\!\!|_W$. Then we see that $\langle x_{k+1} \rangle = V_1$, and an extension $\pi'$ is given by letting $\pi'(u_i) = x_i$.

Let $W \subset \mathbf{V}_\infty$ be a finite dimensional subspace with $V_1\subset W$, and let $<\,\in\!X_{\mathcal{K}^*}$. Define $(N_W)_{V_1}^k$ to be those $<\,\in \!X_{\mathcal{K}^*}$ for which $[V_1, <\!\!|_W] \geq [W,<\!\!|_W]_k$, i.e.\ those orderings for which there are at most $k$ equivalence classes of $W/\!\sim$ greater that $[V_1,<\!\!|_W]$; if $V_1\not\subset W$, set $(N_W)_{V_1}^k = X_{\mathcal{K}^*}$. Note that $(N_U)_{V_1}^k \subseteq (N_W)_{V_1}^k$ for $W\subseteq U$. Hence $$N_{V_1}^k = \!\!\!\bigcap_{W \in \mathrm{Fin}(\mathbf{V}_\infty)} \!\!\!(N_W)_{V_1}^k,$$ 
where $\mathrm{Fin}(\mathbf{K})$ denotes the set of finite substructures of $\mathbf{K}$. In particular, if $V_1 = W_1\subset W_2\subset\cdots$ where each $W_n$ is $n$-dimensional and $\mathbf{V}_\infty = \displaystyle\bigcup_{n=1}^\infty W_n$, we have $N_{V_1}^k = \displaystyle\bigcap_{n=1}^\infty (N_{W_n})_{V_1}^k$. Whether or not $< \,\in (N_W)_{V_1}^k$ depends only on $<\!|_W$; each $(N_W)_{V_1}^k$ is a union of basic open sets of the form $N_{\langle W, <_W\rangle} := \{<\,\in \!X_{\mathcal{K}^*}:\; <\!|_W = <_W\}$. 

Let us compute $\mu((N_W)_{V_1}^k)$. We have:
\vspace{1 mm}\\
$$\mu((N_W)_{V_1}^k) = \frac{\#(<_W\text{ on } W \text{ with } [V_1, <_W] \geq [W,<_W]_k)}{\#(<_W\text{ on } W)}.$$
\vspace{1 mm}\\
Fix $v\in V_1\backslash \{0\}$. For $(v_0,...,v_{m-1})\in F_q$ not all zero, let 
\vspace{-4 mm}\\
\begin{align*}
N_{(v_0,...,v_{m-1})} = \{&<_W:\text{ $<_W$ is given by ordered basis } b_0 >_W\cdots >_W b_{m-1}\\ &\text{and } v = v_0b_0+\cdots v_{m-1}b_{m-1}\}.
\end{align*}
\vspace{-4 mm}\\
First let us show that for any two nonzero $(u_i)_{i=0}^{m-1}$, $(v_i)_{i=o}^{m-1}$, we have $|N_{(u_0,...,u_{m-1})}| = |N_{(v_0,...,v_{m-1})}|$. Let $l$ be least with $v_l$ nonzero. Select $b_0,...,b_{l-1},b_{l+1},...,b_{m-1}$ arbitrarily such that they are linearly independent and don't span $V_1$. There are $(q^m-q)(q^m-q^2)\cdots(q^m-q^{m-1})$ ways to do this. Now set
\vspace{0 mm}\\ 
$$b_l = v_l^{-1}(v - v_0 b_0 -...-v_{l-1} b_{l-1} - v_{l+1} b_{l+1} -...-v_{m-1} b_{m-1}).$$
\vspace{0 mm}\\ 
This ordered basis gives some $<_W\, \in N_{(v_0,...,v_{m-1})}$, and certainly each $<_W$ can be uniquely produced in this manner. Hence $|N_{(v_0,...,v_{m-1})}| = (q^m-q)(q^m-q^2)\cdots(q^m-q^{m-1})$. Now we have
\begin{align*}
\mu((N_W)_{V_1}^k) &= \frac{\sum{\{|N_{(v_0,...,v_{m-1})}|:\text{$(v_i)_{i=0}^{m-1} \neq 0$ and there is $l\leq k$ with $v_l\neq 0$}\}}}{\sum{\{|N_{(v_0,...,v_{m-1})}|:\;(v_i)_{i=0}^{m-1} \neq 0\}}}\\\\
&= \frac{q^m - q^{m-k-1}}{q^m-1} = \frac{1-q^{-k-1}}{1-q^{-m}}.
\end{align*}

Now $\mu(N_{V_1}^k) = \displaystyle\lim_{m\to \infty} \frac{1-q^{-k-1}}{1-q^{-m}} = 1 - \frac{1}{q^{k+1}}$. Letting $k\to \infty$, we have $\mu(N_{V_1}^{fin}) = 1$. This proves Theorem 1.2. \qed \\

As there are countably many $1$-dimensional subspaces of $\mathbf{V}_\infty$, we obtain the following immediate corollary.

\begin{cor}
Let $N_{\omega^*} = \{<\,\in \!X_{\mathcal{K}^*}:\; <\!/\!\sim \,= \omega^*\}$. Then $\mu(N_{\omega^*}) = 1$.
\end{cor}

\begin{proof}
We will show that $N_{\omega^*} = \!\!\!\displaystyle\bigcap_{v\in \mathbf{V}_\infty\backslash \{0\}} \!\!\!N_{\langle v\rangle}^{fin}$. Clearly $N_{\omega^*}\subset \!\!\! \displaystyle\bigcap_{v\in \mathbf{V}_\infty\backslash \{0\}} \!\!\!N_{\langle v\rangle}^{fin}$. To show the other inclusion, note that if $<\,\in \!X_{\mathcal{K}^*}$ with $<\!/\!\sim \,\neq \omega^*$, then there are $[u], [v_i]\in \mathbf{V}_\infty/\!\sim$, $i\in \mathbb{N}$, with $[v_i]\neq[v_j]$ if $i\neq j$ and $[u]<[v_i]$ for each $i$. Pick $u\in [u]$. Then $<\,\notin N_{\langle u\rangle}^{fin}$.
\end{proof}

\section{Matrices of Ordered Inclusion}

In this section we develop some of the tools we will need to prove Theorem 8.3 below. Define a \emph{chain} to be a sequence of subspaces $V_1 \subset V_{2} \subset ...$ with $V_n$ $n$-dimensional and $\bigcup_{m \geq 1}V_m = \mathbf{V}_{\infty}$. Given a chain and an ordering $<\,\in \!X_{\mathcal{K}^*}$, we will write $<_n \,=\, <|_{V_n}$. We will write $B_n = \{b_0^{(n)}>\cdots >b_{n-1}^{(n)}\}$ for the least basis of $<_n$ in $V_n$, i.e.\ the basis which induces the antilexicographic ordering. 

Let $(V_i)_{i\in \mathbb{N}}$ be a chain and $<\,\in \!X_{\mathcal{K}^*}$. For $m > n$, we may represent the inclusion map $i: \langle V_n, <_n\rangle \hookrightarrow \langle V_m, <_m \rangle$ via the change of basis matrix $M_{n, m}$. Writing $M_{n, m} = (m_{ij})$, $0\leq i < m$, $0\leq j < n$, we have
$$b_j^{(n)} = \sum_{i=0}^{n-1}m_{ij}b_i^{(m)}.$$
We will call $M_{n, m}$ as above the \emph{matrix of ordered inclusion}; we will use the shorthand $M_n = M_{n, n+1}$ when there is no confusion. We see that $M_{n, m} = M_{m-1}\dots M_n$. This leads us to ask the following:

\begin{que}
For which $m \times n$ matrices $M$ is there $<\,\in \!X_{\mathcal{K}^*}$ such that $M = M_{n, m}$?
\end{que}
First observe that if $<\,\in \!X_{\mathcal{K}^*}$, then the following must hold:
\begin{enumerate}
\item
For $i < j$, we have $b_j^{(n)} <_m b_i^{(n)}$.
\item
For any $b_i^{(n)}\in B_n$, we have $\forall u \in V_n \left(u <_m b_i^{(n)} \Rightarrow u \ll_m b_i^{(n)}\right) $.
\end{enumerate}
\begin{prop}
Let $B_n$, $B_m$ be any ordered bases of $V_n$, $V_m$ inducing orderings $<_n$, $<_m$. Let $M$ be the matrix of inclusion with respect to these bases. If (1) and (2) hold, then $<_m$ extends $<_n$. 
\end{prop}
\begin{proof}
Let $Z = \{z_0>...>z_{n-1}\}$ be the least basis of $\langle V_n, <_m\!|_{V_n}\rangle$. Consider the following subset of $V_n$:
$$X := \{v\in V_n:\; \forall u\in V_n\left(u <_m v \Rightarrow u \ll_m v\right)\}$$
We see that $Z\subseteq X$ and $B_n\subseteq X$. However, we also see that for $x_1\neq x_2 \in X$, $x_1 \not\sim x_2$. Hence $|X|\leq n$, and $Z = B_n$. It follows that $b_i^{(n)} = z_i$, $0\leq i < n$.
\end{proof}

Call an $m\times n$ matrix $M$ \emph{valid} if $M = M_{n, m}$ for some $<\,\in \!X_{\mathcal{K}^*}$ and some chain. Write $M = (m_{ij})$, $0\leq i < m$, $0\leq j < n$. For each column $j$ of $M$, let $m_j$ denote the least row number with $m_{m_jj}$ nonzero. If $M$ is valid, witnessed by $<\,\in \!X_{\mathcal{K}^*}$ and some chain, then the following must hold:
\begin{itemize}
\item
For $k < l$, we have $m_k < m_l$ since $b_l^{(n)} \ll_m b_k^{(n)}$. 
\item
For each $k$, we have $m_{m_kk} = 1$, as each $b\in B_n$ is minimal in its line.
\item
For $l\neq k$, we have $m_{m_kl} = 0$. For $l > k$, this is clear. For $l < k$, it is because we have $b_l^{(n)} <_m b_l^{(n)} + cb_k^{(n)}$ for all $c\in F_q\backslash \{0\}$.
\end{itemize}

The necessary conditions amount to saying that $M$ is the transpose of a matrix in reduced row echelon form with rank $n$. These conditions are also sufficient; let $M$ satisfy the above, and let $(V_i)_{i\in \mathbb{N}}$ be a chain. Fix an ordered basis $B_n$ of $V_n$, and choose an ordered basis $B_m$ such that $M$ is the matrix of inclusion with respect to these bases; we may do this as $M$ has rank $n$. We will show that the two conditions of Proposition 7.2 are satisfied. The first is clear. Now, for $b_i^{(n)}\in B_n$, suppose $u\in V_n$ with $u <_m b_i^{(n)}$. Write $u = u_0b_0^{(n)} +...+ u_{n-1}b_{n-1}^{(n)}$, and suppose $k$ is least with $u_k\neq 0$; we are done if we can show $k > i$. In the basis $B_m$, we have:
\begin{align*}
u &= \sum_{j=0}^{n-1}\left(u_j\left(\sum_{i=0}^{m-1}m_{ij}b_i^{(m)}\right)\right)\\\\
&= \mu_0b_0^{(m)} +...+ \mu_{m-1}b_{m-1}^{(m)}
\end{align*}
We see that $m_k$ is least with $\mu_{m_k} \neq 0$, so we must have $k\geq i$. Suppose $k=i$. Then $u_k=1$. Let 
\begin{align*}
x = u-b_i^{(n)} &= x_0b_0^{(n)}+...+x_{n-1}b_{n-1}^{(n)}\\
&= \chi_0b_0^{(m)}+...+\chi_{m-1}b_{m-1}^{(m)}
\end{align*}
and let $j$ be least with $x_j \neq 0$. We have that $j>i$. Then $m_j$ is least with $\chi_{n_j}\neq 0$. But $m_{m_ji}=0$, from which it follows that $u >_m b_i^{(n)}$, a contradiction. 

The case $m = n+1$ will be of special interest, so let us introduce some terminology specific to this case. For $M$ a valid matrix, the map $i \rightarrow m_i$ has range which excludes a single number $k$, $0\leq k\leq n$. Call such a matrix type $k$. Denote the type of matrix $M$ by $t(M)$. A natural question to ask is how many valid $(n+1)\times n$ matrices $M$ have $t(M)=k$. By the necessary conditions above, we see that $m_{ij}$ is determined except for those pairs $(i,j)$ with both $i = k$ and $j < k$; for these values of $i$ and $j$, any choice of $m_{ij}\in F_q$ gives us a valid $M$. Hence there are $q^k$ valid matrices $M$ of type $k$. Observe that $[V_n, <_n]_k\subset [V_{n+1},<_{n+1}]_k$ if $t(M_n) > k$, and $[V_n, <_n]_k\subset [V_{n+1}, <_{n+1}]_{k+1}$ if $t(M_n)\leq k$. In particular, $t(M_n) = k$ iff $[V_{n+1}, <_{n+1}]_k\cap V_n = \emptyset$. 

Let $\mathcal{M}_n$ be the set of valid $(n+1)\times n$ matrices, which we will equip with the uniform probability measure $\rho_n$. Let $\mathcal{M} = \prod_{n\in \mathbb{N}}{\mathcal{M}_n}$, and let $\rho$ be the product measure. Fix a chain $(V_i)_{i\in \mathbb{N}}$. There is a surjection $\pi:\: X_{\mathcal{K}^*}\rightarrow \mathcal{M}$ as follows; for any $<\,\in \!X_{\mathcal{K}^*}$, let $\pi(<) = (M_i)_{i\in \mathbb{N}}$, where $M_i$ is the matrix of inclusion for the chain $(V_i)_{i\in \mathbb{N}}$ and the ordering $<$.

\begin{prop}
Let $\mu$ be the unique measure on $X_{\mathcal{K}^*}$. Then $\rho = \pi_*\mu$.
\end{prop}
\begin{proof}
First we will show the following: let $<_n$ be a natural order on $V_n$, and let $M\in \mathcal{M}_n$. First let us show that the number of natural $<_{n+1}$ on $V_{n+1}$ extending $<_n$ such that the matrix of ordered inclusion is $M$ does not depend on $M$ or $<_n$. Write $M = (m_{ij})$, $0\leq i < n+1$, $0\leq j < n$. Pick $n$ linearly independent rows of $M$; without loss of generality assume rows $0\leq i < n$ are independent. We now have the following procedure for producing $<_{n+1}$ extending $<_n$: first pick $b_n^{(n+1)}$ from among the $q^{n+1}-q^n$ vectors in $V_{n+1}\backslash V_n$. Now set the other $b_i^{(n+1)}$ to be the unique solution to the system of equations
\begin{align*}
\sum_{i=0}^{n-1}{m_{ij}b_i^{(n+1)}} = b_j^{(n)} - m_{nj}b_n^{(n+1)} \qquad (0\leq j < n).
\end{align*}
This procedure can produce any natural $<_{n+1}$ extending $<_n$ with matrix of ordered inclusion $M$. Moreover, we see that there are $q^{n+1}-q^n$ such extensions.

Let $U(S_1,...,S_k) = S_1\times...\times S_k \times \prod_{n>k}{\mathcal{M}_k}$ be a basic open set in $\mathcal{M}$. Note that whether or not an ordering $<$ is in $\pi^{-1}(U)$ depends only on $<\!|_{V_{k+1}}$. Now we have
\begin{align*}
\pi_*\mu((U)) &= \pi_*\mu(U(S_1)\cap U(\mathcal{M}_1,S_2)\cap...\cap U(\mathcal{M}_1,...,\mathcal{M}_{k-1},S_k))\\
&= \pi_*\mu(U(S_1))*\pi_*\mu(U(\mathcal{M}_1,S_2))*...*\pi_*\mu(U(\mathcal{M}_1,...,\mathcal{M}_{k-1},S_k))\\
&= (|S_1|/|\mathcal{M}_1|)*...*(|S_k|/|\mathcal{M}_k|)\\
&= \rho(U(S_1,...,S_k)).
\end{align*}
\end{proof}  

\section{A Representation of $\mu$}

To conclude this paper, we provide a more concrete representation of the measure $\mu$. First, we need a few general lemmas.
\begin{lemma}
Let $\langle V, <\rangle$ have the FLP, and fix $v\in V$. Suppose $u, w\in V\setminus \{0\}$ are minimal in their lines with $u\sim w$. Let $v_u\in F_q$ be such that $v-v_uu < v-du$ for $d\neq c_u$, and likewise for $w$. Then $v_u = v_w$.
\end{lemma}
\begin{proof}
Form $U = \langle u, v, w\rangle$, and let $<\!|_U$ be induced by basis $\{x_0,...,x_l\}$, $l\leq 2$. Write $u = u_0x_0 +...+u_lx_l$, etc. Let $k$ be the least number with $u_k$ nonzero. Then $k$ is also the least number with $w_k$ nonzero. We have $u_k = w_k = 1$; we see that $c_u = c_w = v_k$.
\end{proof}
It now makes sense to define $v_{[u]}$ as in lemma 8.1 for $v\in V$ and $[u]\in V/\sim$.
\begin{lemma}
For any $u,v,w\in V\setminus \{0\}$ and $d\in F_q\setminus \{0\}$, we have $u_{[w]}+v_{[w]} = (u+v)_{[w]}$ and $(du)_{[w]} = d(u_{[w]})$.
\end{lemma}
\begin{proof}
Let $u,v\in V$. Form $U = \langle u, v, w\rangle$, and let $<\!|_U$ be induced by basis $\{x_0,...,x_l\}$, $l\leq 2$. Write $u = u_0x_0 +...+u_lx_l$, etc. Let $k$ be the least number with $w_k$ nonzero. Then $w_k = 1$; we see that $u_{[w]} = u_k$, $v_{[w]} = v_k$, $(u+v)_{[w]} = (u+v)_k = u_k + v_k$, and $(du)_{[w]} = du_k = d(u_{[w]})$. 
\end{proof}
We may now injectively map $\langle V, <\rangle$ to a subspace of $\langle F_{q}^{V/\sim}, \prec \rangle$, where $\prec$ is a partial ordering with $\alpha \prec \beta$ iff for some $[u]$ in $V/\!\sim$, we have $\alpha([u]) < \beta([u])$ and $\alpha([v]) \leq \beta([v])$ for all $[v] > [u]$. We will be most interested in the case $V/\!\sim \,= \omega^*$; in this case $\prec$ is a linear order, as $V/\!\sim$ has no infinite ascending chains. For a fixed $< \,\in \!N_{\omega^*}$, enumerate $\mathbf{V}_{\infty}/\sim$ by $\alpha_0 >\alpha_1 >\cdots$. Now for any $<\,\in \!N_{\omega^*}$ and $v\in \mathbf{V}_{\infty}$, we may identify $v\in F_{q}^{\omega^*}$. To make this identification explicit, pick $w_i\!\in\! \alpha_i$ minimal in their lines, and define $\phi_<:\: \langle \mathbf{V}_\infty, <\rangle \rightarrow \langle F_{q}^{\omega^*}, \prec\rangle$ via $\phi_<(v) = (v_i)_{i\in \omega^*}$, where $v_i = v_{[w_i]}$.

Fix any basis $B = \{b_0, b_1,...\}$ of $\mathbf{V}_\infty$. Define $\phi:\: N_{\omega^*} \rightarrow (F_q^{\omega^*})^{\omega^*}$ by setting $\phi(<) = (\phi_<(b_i))_{i\in \omega^*}$. Note that $\phi$ is a Borel map. Equip $(F_q^{\omega^*})^{\omega^*}$ with the product measure $\sigma$. 
\begin{theorem}
The map $\phi$ is injective and a.e.\ surjective. Moreover, $\sigma = \phi_*\mu$, giving a mod zero isomorphism of $(X_{\mathcal{K}^*}, \mu)$ and $((F_q^{\omega^*})^{\omega^*}, \sigma)$.
\end{theorem}
\begin{proof}\renewcommand{\qedsymbol}{}
To see injectivity, it suffices to note that for any $\beta \in \mathrm{Im}(\phi)$ and any $<\,\in \phi^{-1}(\beta)$, the map $\phi_<$ is completely determined, which in turn determines $<$.

To show that $\phi$ is a.e.\ surjective, consider $\beta = (\beta_i)_{i\in \omega^*}\in (F_q^{\omega^*})^{\omega^*}$. Certainly $\beta\in \mathrm{Im}(\phi)$ if the following hold:
\begin{enumerate}
\item
The $\beta_i$ are linearly independent.
\item
For each $k>0$, there is an $i$ with $\beta_i|_k = 0\char94\cdots\char94 0\char94 c$, $c \neq 0$.
\end{enumerate}
The second condition is easily seen to be the countable intersection of measure $1$ conditions. For the first condition, observe that this is the countable intersection of conditions $c_0\beta_{i_0}+\cdots +c_k\beta_{i_k} \neq 0$, each of which is measure 1.
\end{proof}
To show that $\sigma = \phi_*\mu$, it suffices to prove the next lemma. For $v_0,...,v_{n-1}\in \mathbf{V}_{\infty}$ and $s_0,...s_{n-1}\in F_{q}^k$, define:
$$N(v_0, s_0,...,v_{n-1}, s_{n-1}) = \{<\,\in \!N_{\omega^*}:\; v_i = s_i\char94 \beta_i \text{ for some } \beta_i\in F_{q}^{\omega^*}\}$$
\begin{lemma}
For $v_0,...,v_{n-1}$ linearly independent, $\mu(N(v_0, s_0,...,v_{n-1}, s_{n-1})) = q^{-kn}$.
\end{lemma}
\begin{proof}
Fix $<\,\in \!N_{\omega^*}$ and a chain $(V_i)_{i\in \mathbb{N}}$, and let $\pi$ be as in Proposition 5. First we consider the case where the $s_i$ are linearly independent. The probability that $v_i = s_i\char94 \beta_i$ for $0\leq i < n$ is bounded below by the probability that the following both occur for some $l$ with $v_0,...,v_{n-1}\in V_l$:
\begin{enumerate}
\item
Let $<_l$ on $V_l$ be given by basis $\{x_0,...,x_{l-1}\}$. Write $v_i = a_0^ix_0 +...+ a_{l-1}^ix_{l-1}$. Then $a_0^i\char94 ...\char94  a_{k-1}^i = s_i$.
\item
Now suppose $\pi(<) = (M_i)_{i\in \mathbb{N}} \in \mathcal{M}$. Then we have $t(M_i) \geq k$ for $i \geq l$.
\end{enumerate}
Call the first event $A_1(l)$ and the second event $A_2(l)$; these events are independent. For linearly independent $(a_0^i,...,a_{l-1}^i)$, $0\leq i < n$, the number of ordered bases $Y = \{y_0,...,y_{l-1}\}$ with $v_i = a_0^iy_0 +...+ a_{l-1}^iy_{l-1}$, $0\leq i < n$, does not depend on which particular linearly independent $(a_0^i,...,a_{l-1}^i)$ are being considered. Therefore:
\begin{align*}
\mathbf{P}(A_1(l)) &= \frac{\#(\text{linearly independent } (a_0^i,...,a_{l-1}^i)\text{ with } a_0^i\char94...\char94 a_{k-1}^i = s_i)}{\# (\text{linearly independent } (a_0^i,...,a_{l-1}^i))}\\
\\
&= \frac{q^{n(l-k)}}{(q^l-1)(q^l-q)...(q^l-q^{n-1})} 
\end{align*}
Now for $i \geq l$, let $B_i(k)$ be the event that $t(M_i) \geq k$. We have $\mathbf{P}(B_i(k)) = \frac{q^{i+1}-q^{k}}{q^{i+1}-1}$. The events $B_i(k)$ are mutually independent, hence:
\begin{align*}
\mathbf{P}(A_2(l)) = &\lim_{m \to \infty} \prod_{l\leq i < m}{\mathbf{P}(B_i(k))}\\
= &\lim_{m \to \infty} \frac{q^{(m-l)(k)}(q^{l-k+1}-1)...(q^l-1)}{(q^{m-k+1}-1)...(q^m-1)}\\
\geq &\:(1-q^{k-1-l})^{k}
\end{align*}
Now we have:
$$\lim_{l \to \infty} \mathbf{P}(A_1(l))\cdot\mathbf{P}(A_2(l)) = q^{-nk}$$
It follows that $\mu(N(v_0, s_0,...,v_{n-1}, s_{n-1})) \geq q^{-nk}$. Equality follows since for $(t_0,...,t_{n-1})\neq (s_0,...,s_{n-1})$, we have $N(v_0, s_0,...,v_{n-1}, s_{n-1}) \cap N(v_0,t_0,...,v_{n-1},t_{n-1}) = \emptyset$.   

When the $s_i$ are not linearly independent, let:
$$L^m = \{(t_0,...,t_{n-1}):\; t_i\in F_q^m \text{ and the } t_i \text{ are linearly independent}\}$$
We have the lower bound:
$$\mu(N(v_0, s_0,...,v_{n-1}, s_{n-1})) \geq \!\!\!\sum_{(t_0,...t_{n-1})\in L^m} \!\!\!{N(v_0, s_0\char94 t_0,..., v_{n-1}, s_{n-1}\char94 t_{n-1})}$$
We have $|L^m| = (q^m-1)(q^m-q)...(q^m-q^{n-1})$, giving us:
$$\mu(N(v_0, s_0,...,v_{n-1}, s_{n-1})) \geq q^{-k(n+m)}(q^m-1)...(q^m-q^{n-1})$$ for any $m$. Letting $m\to \infty$, we see that $\mu(N(v_0, s_0,...,v_{n-1}, s_{n-1})) \geq q^{-nk}$, and hence $\mu(N(v_0, s_0,...,v_{n-1}, s_{n-1})) = q^{-nk}$. 

\end{proof}

\section{Questions and Further Work}

Our investigations above lend themselves to a number of open questions:

\begin{que}
Are there other examples where Theorem 3.1 can be used to show non-amenability? Are there any examples where Theorem 3.1 can be used to show amenability?
\end{que}

\begin{que}
Assume the Fra\"iss\'e class $\mathcal{K}$ admits a companion $\mathcal{K}^*$. Let $\mathbf{K}$ be the Fra\"iss\'e limit of $\mathcal{K}$. If $\mathrm{Aut}(\mathbf{K})$ is uniquely ergodic, what are necessary and sufficient conditions for the unique measure on any minimal flow to be supported on the generic orbit?
\end{que} 

Pongr\'acz in [P] has given a partial answer to Question 8.2. Let $L^* = L\cup \{<\}$, for $<$ a symbol for a linear ordering and $L$ \textbf{relational}. Let $(\mathcal{K}, \mathcal{K}^*)$ be an excellent pair of Fra\"iss\'e classes in $L$ and $L^*$. Suppose $\mathcal{K}^*$ is order forgetful, i.e.\ for $\langle \mathbf{A}, <\rangle, \langle \mathbf{B}, <'\rangle \in \mathcal{K}^*$, we have $\mathbf{A} \cong \mathbf{B} \Leftrightarrow \langle \mathbf{A}, <\rangle \cong \langle \mathbf{B}, <'\rangle$. We see that $\mathrm{Aut}(\mathbf{K})$, if amenable, is uniquely ergodic, and the measure satisfies $\mu(N_{\langle \mathbf{A}, <\rangle}) = 1/k_{\mathbf{A}}$ (see the introduction). Pongr\'acz has shown that in this case, $\mu$ is supported generically. Note that this does not contradict Theorem 1.2; every hypothesis of Pongr\'acz's theorem is satisfied except that the language of vector spaces contains function symbols. His calculations also shed some light on the role that functions play in my calculations for $\mathbf{V}_\infty$, and they also suggest that we may be able to find relational examples with $L^* = L\cup \{S_1,...,S_n\}$ with the measure $\mu$ as above not supported generically. 

\begin{que} 
When the unique measure is not supported generically, where is it supported?
\end{que}

To make question 9.3 more precise, consider what was shown in sections 7 and 8. Though no single orbit has positive measure, it seems that by taking a suitable completion of $\mathbf{V}_\infty$, the unique measure concentrates on the isomorphism type of $F_q^{\omega^*}$ ordered lexicographically. In what sense can this be made precise and generalized to other structures?

\vspace{5 mm}

Andy Zucker

Carnegie Mellon University

Dept. of Mathematical Sciences

Pittsburgh, PA 15213

zucker.andy@gmail.com

\end{document}